\documentclass[11pt]{amsart}

\usepackage[a4paper,hmargin=3.5cm,vmargin=4cm]{geometry}
\usepackage{amsfonts,amssymb,amscd,amstext}
\usepackage{graphicx}
\usepackage[dvips]{epsfig}

\usepackage{fancyhdr}
\input xy
\xyoption{all}

\usepackage{times}

\usepackage{enumerate}
\usepackage{titlesec}
\usepackage{mathrsfs}

\titleformat{\section}
{\filcenter\bfseries\large} {\thesection{.}}{0.2cm}{}
\titleformat{\subsection}[runin]
{\bfseries} {\thesubsection{.}}{0.15cm}{}[.]
\titleformat{\subsubsection}[runin]
{\em}{\thesubsubsection{.}}{0.15cm}{}[.]

\usepackage[up,bf]{caption}

\newtheorem{theorem}{Theorem}[section]

\newtheorem{lemma}[theorem]{Lemma}
\newtheorem{corollary}[theorem]{Corollary}

\theoremstyle{definition}
\newtheorem{definition}[theorem]{Definition}
\newtheorem{remark}[theorem]{Remark}

\numberwithin{equation}{section}
\numberwithin{figure}{section}

\newcommand\Ascr{\mathscr{A}}

\newcommand\Cscr{\mathscr{C}}

\newcommand\Lscr{\mathscr{L}}
\newcommand\Oscr{\mathscr{O}}

\renewcommand\b{\mathbb{B}}
\renewcommand\c{\mathbb{C}}

\newcommand\cp{\mathbb{CP}}

\newcommand\h{\mathbb{H}}
\newcommand\n{\mathbb{N}}
\renewcommand\r{\mathbb{R}}

\newcommand\z{\mathbb{Z}}

\newcommand\igot{\mathfrak{i}}

\renewcommand\igot{\mathfrak{i}}

\newcommand\Ygot{\mathfrak{Y}}

\renewcommand\imath{\igot}

\newcommand\hra{\hookrightarrow}

\newcommand\wt{\widetilde}
\newcommand\wh{\widehat}
\newcommand\di{\partial}

\newcommand\cM{\overline{M}}

\usepackage{color}

\begin{document}

\fancyhead[LO]{Proper holomorphic Legendrian curves in $SL_2(\c)$}
\fancyhead[RE]{A.\ Alarc\'on}
\fancyhead[RO,LE]{\thepage}

\thispagestyle{empty}

\vspace*{1cm}
\begin{center}
{\bf\LARGE Proper holomorphic Legendrian curves in $SL_2(\c)$}

\vspace*{5mm}

{\large\bf Antonio Alarc\'on}
%
%
\end{center}


\vspace*{1cm}

\begin{quote}
{\small
\noindent {\bf Abstract}\hspace*{0.1cm}
In this paper we prove that every open Riemann surface properly embeds in the Special Linear group $SL_2(\c)$ as a holomorphic Legendrian curve, where $SL_2(\c)$ is endowed with its standard contact structure.
As a consequence, we derive the existence of proper, weakly complete, flat fronts in the real hyperbolic space $\h^3$ with arbitrary complex structure.

\smallskip

\noindent{\bf Keywords}\hspace*{0.1cm} Complex contact manifolds, Riemann surfaces, holomorphic Legendrian curves, flat fronts.

\smallskip

\noindent{\bf MSC (2010)}\hspace*{0.1cm} 53D10, 32B15, 32E30, 32H02, 53C42, 53A35}
\end{quote}


\section{Introduction and main results} 
\label{sec:intro}

Let $n\in\n=\{1,2,3,\ldots\}$ be a positive integer. A {\em complex contact manifold} is a complex manifold $W$ of odd dimension $2n+1\ge 3$ endowed with a {\em holomorphic contact structure} $\Lscr$. The latter is a holomorphic vector subbundle $\Lscr\subset TW$ of complex codimension one in the tangent bundle $TW$, satisfying that every point $p\in W$ admits an open neighborhood $U\subset W$ such that 
\[
	\Lscr|_U=\ker\eta
\]
for a holomorphic $1$-form $\eta$ on $U$ which satisfies
\[
	\eta\wedge(d\eta)^n=\eta\wedge d\eta\wedge \stackrel{\text{$n$ times}}{\cdots}\wedge d\eta
	\neq 0\quad \text{everywhere on $U$}.
\]
We shall write $(W,\eta)$ instead of $(W,\Lscr)$ when the defining holomorphic contact $1$-form $\eta$ is globally defined on $W$; in such case the complex contact manifold is said to be {\em strict}.

The model example of a complex contact manifold is the complex Euclidean space $\c^{2n+1}$ endowed with its {\em standard holomorphic contact form}
\begin{equation}\label{eq:contact-C2n+1}
	\eta_0=dz+\sum_{j=1}^n x_jdy_j,
\end{equation}
where $(x_1,y_1,\ldots,x_n,y_n,z)$ denote the complex coordinates on $\c^{2n+1}$. Another example of a complex contact manifold which is the focus of interest is the Special Linear group
\[
	SL_2(\c)=\left\{z=\left(
	\begin{array}{cc}
	z_{11} & z_{12}
	\\
	z_{21} & z_{22}
	\end{array}\right)\colon \det z=z_{11}z_{22}-z_{21}z_{12}=1\right\}
\]
endowed with its {\em standard holomorphic contact form}
\[
	\eta_{SL}=z_{11}dz_{22}-z_{21}dz_{12}.
\]
From now on in this paper we will just write $\c^{2n+1}$ and $SL_2(\c)$ for the complex contact manifolds $(\c^{2n+1},\eta_0)$ and $(SL_2(\c),\eta_{SL})$, respectively.

Let $(W,\Lscr)$ be a complex contact manifold and $M$ be a complex manifold. A holomorphic map $F\colon M\to W$ is said to be {\em Legendrian} if it is everywhere tangent to the contact structure:
\[
	dF_p(T_pM)\subset\Lscr_{F(p)}\quad \text{for all $p\in M$}.
\]
If $\Lscr=\ker\eta$ for a holomorphic contact $1$-form $\eta$, then $F$ is Legendrian if, and only if,
\[
	F^*\eta=0.
\]
When $M$ is an open Riemann surface, a holomorphic Legendrian map $F\colon M\to W$ is called a {\em Legendrian curve}; the same definition applies when $M$ is a {\em compact bordered Riemann surface} (i.e. a compact Riemann surface with nonempty boundary consisting of finitely many pairwise disjoint smooth Jordan curves; its interior $\mathring M=M\setminus bM$ is called a {\em bordered Riemann surface}) and $F$ is a Legendrian map of class $\Ascr^1(M)$ (i.e. of class $\Cscr^1(M)$ and holomorphic in the interior $\mathring M=M\setminus bM$).

A major problem in complex contact geometry is to determine whether a given complex contact manifold admits properly embedded holomorphic Legendrian curves with arbitrary complex structure. Thus, in the compact case, a celebrated result by Bryant from 1982 (see \cite[Theorem G]{Bryant1982JDG}) ensures that every compact Riemann surface embeds as a holomorphic Legendrian curve in the complex projective space $\cp^3$ endowed with the holomorphic contact form obtained by projectivizing the standard symplectic form of $\c^4$ (this is in fact, up to contactomorphisms, the only holomorphic contact structure in $\cp^3$; see LeBrun \cite{LeBrun1995IJM}).  In the open case, it has been a long-standing open problem, positively settled only very recently by Forstneri\v c, L\'opez, and the author, whether every open Riemann surface $M$ admits a proper holomorphic Legendrian embedding $M\hra\c^3$ (see \cite[Theorem 1.1]{AlarconForstnericLopez2016Legendrian}). In the opposite direction, Forstneri\v c \cite{Forstneric2016Legendrian} has recently proved that, for every $n\in\n$, there exists a holomorphic contact form $\eta_F$ on $\c^{2n+1}$ such that the complex contact manifold $(\c^{2n+1},\eta_F)$ is Kobayashi hyperbolic; in particular, any holomorphic $\eta_F$-Legendrian curve $\c\to\c^{2n+1}$ is constant. 
The aim of this paper is to settle the embedding problem for holomorphic Legendrian curves in $SL_2(\c)$.
%
%
\begin{theorem}\label{th:intro-SL2C}
Every open Riemann surface $M$ admits a proper holomorphic Legendrian embedding $M\hra SL_2(\c)$.
\end{theorem}
The holomorphic map $\Ygot\colon \c^3\to SL_2(\c)$ given by
\begin{equation}\label{eq:Ygot}
	\Ygot(x,y,z)=\left(
	\begin{array}{cc}
	e^{-z} & xe^{z}
	\\
	ye^{-z} & (1+xy)e^z
	\end{array}\right),\quad (x,y,z)\in\c^3,
\end{equation}
maps holomorphic Legendrian immersions $M\to\c^3$ into holomorphic Legendrian immersions $M\to SL_2(\c)$ (see Mart\'in, Umehara, and Yamada \cite[p.\ 210]{MartinUmeharaYamada2014RMI}). However, there are two main problems with using the map $\Ygot\colon\c^3\to SL_2(\c)$ in order to obtain properly embedded Legendrian curves in $SL_2(\c)$; namely, it is neither injective nor proper (see Remark \ref{rem:C3->SL2C} for a careful discussion of this assertion). 
The proof of Theorem \ref{th:intro-SL2C} consists of, given an open Riemann surface $M$, constructing a holomorphic Legendrian embedding $F=(X,Y,Z)\colon M\hra \c^3$ such that 
\[
	(X,Y,e^Z)\colon M\to\c^3\quad \text{is one-to-one}
\]
and
\[
	\max\{|X|,-\Re Z\}\colon M\to\r_+=[0,+\infty)\quad \text{is a proper map},
\]
where $\Re$ denotes real part; see Theorem \ref{th:main-SL2C} for a more precise statement. It easily follows from these two conditions that $\Ygot\circ F\colon M\hra SL_2(\c)$, where $\Ygot$ is the map \eqref{eq:Ygot}, is a proper holomorphic Legendrian embedding. 
%
Forstneri\v c, L\'opez, and the author made in \cite{AlarconForstnericLopez2016Legendrian} a systematic investigation of holomorphic Legendrian curves in the complex Euclidean spaces; in particular, they proved several approximation results which will be exploited in the present paper.

Theorem \ref{th:intro-SL2C} is in connection with a result by Forstneri\v c and the author asserting that every bordered Riemann surface properly embeds in $SL_2(\c)$ as a holomorphic {\em null} curve (see \cite[Corollary 1.5]{AlarconForstneric2015MA}). The latter is a holomorphic immersion $F\colon M\to SL_2(\c)$ of an open Riemann surface $M$ into $SL_2(\c)$ satisfying the nullity condition $\det dF=0$ everywhere on $M$.  It still remains open the question whether every open Riemann surface properly embeds (or at least immerses) in $SL_2(\c)$ as a holomorphic {\em null} curve (cf.\ \cite[Problem 1, p.\ 919]{AlarconForstneric2015MA}). 
The main difference between {\em Legendrian} and {\em null} curves is that the holomorphic distribution controlling Legendrian curves does depend on the base point, and hence the constructions in the Legendrian case become more delicate and involved.

Theorem \ref{th:intro-SL2C} is also related to the embedding problem for open Riemann surfaces in $\c^2$, asking whether every such properly embeds in $\c^2$ as a complex curve (see Forstneri\v c and Wold \cite{ForstnericWold2009JMPA,ForstnericWold2013APDE} and the references therein for a discussion of the state of the art of this long-standing, likely very difficult, open problem).

A {\em flat front}  in the hyperbolic space $\h^3$ is a flat surface in $\h^3$ with admissible singularities. Here by a {\em flat surface} we mean a surface in $\h^3$ whose Gauss curvature vanishes at the regular points; a singular point of such surface is said an {\em admissible singularity} if the corresponding points on nearby parallel surfaces are regularly immersed (see e.g. Kokubu, Umehara, and Yamada \cite[\textsection2]{KokubuUmeharaYamada2004PJM} for more details). It is classical that the only complete smooth flat surfaces in $\h^3$ are the horospheres and the hyperbolic cylinders (see Sasaki \cite{Sasaki1973KMSR}), and this is why the study of flat surfaces with singularities in $\h^3$, and, in particular, of flat fronts, has been the focus of interest in this subject during the last decades (see e.g. \cite{GalvezMartinezMilan2000MA,KokubuUmeharaYamada2004PJM,GalvezMira2005CVPDE,KokubuRossmanUmeharaYamada2007JMSJ,MartinUmeharaYamada2014RMI,MartinezMilan2014AGAG,MartinezRoitmanTenenblat2015AGAG} and the references therein). Flat fronts in $\h^3$ enjoy a nice theory; it is for instance well known that these objects admit a Weierstrass type representation formula in terms of holomorphic data on a Riemann surface (see G\'alvez, Mart\'inez, and Mil\'an \cite{GalvezMartinezMilan2000MA} and G\'alvez and Mira \cite{GalvezMira2005CVPDE}). It is also well known that, given an open Riemann surface $M$ and a holomorphic Legendrian immersion $F\colon M\to SL_2(\c)$, the map
\[
	F \bar F^t\colon M\to\h^3=SL_2(\c)/SU(2)=\{a \bar a^t\colon a\in SL_2(\c)\},
\]
where $\bar\cdot$ and $\cdot^t$ denote complex conjugation and transpose matrix, respectively, determines a flat front in $\h^3$ which is conformal with respect to the metric induced on $M$ by the second fundamental form of $F\bar F^t$, and that this fact locally characterizes flat fronts in $\h^3$ (see \cite{KokubuUmeharaYamada2004PJM,GalvezMira2005CVPDE}). Moreover, the flat front $F\bar F^t$ is said to be {\em weakly complete} (according to Kokubu, Rossman, Umehara, and Yamada \cite[\textsection3]{KokubuRossmanUmeharaYamada2007JMSJ}) if its holomorphic lift $F$ is complete in the sense that the Riemannian metric induced on $M$ by the one in $SL_2(\c)$ via $F$ is complete. Properness and weakly completeness are the most natural global assumptions in the theory of flat fronts in $\h^3$.

Since the map 
\[
	SL_2(\c)\ni a\longmapsto a\bar a^t\in\h^3
\]
is proper, Theorem \ref{th:intro-SL2C} implies the following
%
%
\begin{corollary}\label{co:intro-H3}
Every open Riemann surface $M$ is the complex structure associated to the second fundamental form of a proper, weakly complete, flat front in $\h^3$. 
\end{corollary}
If $M$ is an open Riemann surface and there is a {\em complete} flat front $M\to\h^3$ (in the sense of \cite{KokubuUmeharaYamada2004PJM,KokubuRossmanUmeharaYamada2007JMSJ}) being conformal with respect to the second fundamental form, then there are a compact Riemann surface $M'$ and a finite subset $\{p_1,\ldots,p_m\}\subset M'$ such that $M$ is biholomorphic to $M'\setminus\{p_1,\ldots,p_m\}$ (see again \cite{KokubuUmeharaYamada2004PJM,KokubuRossmanUmeharaYamada2007JMSJ}). Thus, completeness imposes strong restrictions on the complex structure, and even on the topology, of flat fronts in $\h^3$; in particular the examples given in Corollary \ref{co:intro-H3} are not complete in general.

\subsection*{Organization of the paper}

In Section \ref{sec:prelim} we state the notation and preliminaries that will be needed throughout the paper. Section \ref{sec:GP} is devoted to the proof of a general position result for holomorphic Legendrian curves in $\c^3$ which will be the key to ensure the embeddedness of the examples in Theorem \ref{th:intro-SL2C}. Finally, we prove Theorem \ref{th:intro-SL2C} in Section \ref{sec:SL2C}.


\section{Preliminaries}\label{sec:prelim}

We denote by $\c_*=\c\setminus\{0\}$, $\imath=\sqrt{-1}$, $\z_+=\{0,1,2,\ldots\}$, and $\r_+=[0,+\infty)$.
Given $N\in\n$, we denote by $|\cdot|$ 
the Euclidean norm 
in $\c^N$. 
Let $K$ be a compact topological space and $f\colon K\to \c^N$ be a continuous map, we denote by 
\[
	\|f\|_{0,K}:=\max\{|f(p)|\colon p\in K\}
\]
the maximum norm of $f$ on $K$. Likewise, if $K$ is a subset of a Riemann surface $M$, then for any $r\in\z_+$ we denote by 
\[
	\|f\|_{r,K}
\]
the standard $\Cscr^r$-norm of a function $f\colon K\to\c^n$ of class $\Cscr^r(K)$, where, if $r>0$, the derivatives are measured with respect to any fixed Riemannian metric on $M$ (its precise choice will not be important in the paper).

%
%

Let $M$ be an open Riemann surface. Given a subset $A\subset M$ we denote by $\Oscr(A)$ the space of functions $A\to \c$ which are holomorphic on an unspecified open neighborhood (depending on the map) of $A$ in $M$. If $A\subset M$ is a smoothly bounded compact domain and $r\in\z_+$, we denote by $\Ascr^r(A)$ the space of $\Cscr^r$ functions $A\to\c$ which are holomorphic on the interior $\mathring A=A\setminus bA$; we just write $\Ascr^r(A)$ for $(\Ascr^r(A))^N=\Ascr^r(A)\times \stackrel{\text{$N$ times}}{\cdots}\times\Ascr^r(A)$ when there is no place for ambiguity. Thus, by a Legendrian curve $A\to\c^{2n+1}$ $(n\in\n)$ of class $\Ascr^r(A)$, $r\ge 1$, we simply mean a map of class $\Ascr^r(A)$ whose restriction to $\mathring A$ is a holomorphic Legendrian curve. 

A {\em compact bordered Riemann surface} is a compact Riemann surface $\cM$ with nonempty boundary $bM\subset\cM$ consisting of finitely many pairwise disjoint smooth Jordan curves. The interior $M=\cM\setminus bM$ of $\cM$ is called a {\em bordered Riemann surface}. It is classical that every compact bordered Riemann surface $\cM$ is diffeomorphic to a smoothly bounded compact domain in an open Riemann surface $M'$. The space $\Ascr^r(\cM)$ is defined as above.

%
%
\subsection{A Mergelyan theorem for Legendrian curves}\label{ss:Mergelyan}

A compact subset $K$ of an open Riemann surface $M$ is said to be {\em Runge}, or {\em holomorphically convex}, if $M\setminus K$ has no relatively compact connected components in $M$. By Mergelyan's theorem, $K\subset M$ is Runge if, and only if, every continuous function $K\to\c$, holomorphic in $\mathring K$, may be approximated uniformly on $K$ by holomorphic functions $M\to\c$ (see \cite{Runge1885AM,Mergelyan1951DAN}).

%
%
\begin{definition}[\text{\cite[Def.\ 4.2]{AlarconForstnericLopez2016Legendrian}}]\label{def:admissible}
A compact subset $S$ of an open Riemann surface $M$ is called {\em admissible} if $S=K\cup\Gamma$, where $K=\bigcup_j\overline D_j$ is a union of finitely many pairwise disjoint, smoothly bounded, compact domains $\overline D_j$ in $M$ and $\Gamma=\bigcup_i\Gamma_i$ is a union of finitely many pairwise disjoint smooth arcs or closed curves in $M$ that intersect $K$ only in their endpoints (or not at all), and such that their intersections with the boundary $bK$ are transverse.
\end{definition}

Let $S=K\cup\Gamma\subset M$ be an admissible subset of an open Riemann surface and $r\in\z_+$ be a nonnegative integer. We denote by
\[
	\Ascr^r(S)=\Cscr^r(S)\cap\Oscr(\mathring K),
\]
and endow the space $\Ascr^r(S)$ with the natural $\Cscr^r(S)$-topology which coincides with the $\Cscr^r(K)$-topology on the subset $K$ and with the $\Cscr^r$-norm of the function measured with respect to a fixed regular parametrization of $\Gamma_i$ on each of the arcs $\Gamma_i\subset\Gamma$. 

Let $\theta$ be a holomorphic $1$-form vanishing nowhere on $M$ (such always exists by the Oka-Grauert principle (see \cite[Theorem 5.3.1]{Forstneric2011book}); see also \cite{AlarconFernandezLopez2012CMH} for an alternative proof). Given a function $f\colon S\to\c$ of class $\Ascr^1(S)$ we define
\begin{equation}\label{eq:df}
	df:=\wh f\theta,
\end{equation}
where $\wh f\colon S\to\c$ is the function of class $\Ascr^0(S)$ given by $\wh f=df/\theta$ on $K$ and $\wh f(\alpha(t))=(f\circ\alpha)'(t)/\theta(\alpha(t),\dot\alpha(t))$ for any smooth regular path $\alpha$ in $M$ parametrizing a connected component $\Gamma_i$ of $\Gamma$. Obviously, $\wh f$ depends on the choice of $\theta$, but $df$ does not.

%
%
\begin{definition}[\text{\cite[Def.\ 4.2]{AlarconForstnericLopez2016Legendrian}}]\label{def:generalized}
Let $S=K\cup\Gamma$ be an admissible subset of an open Riemann surface $M$. A map $F=(X_1,Y_1,\ldots,X_n,Y_n,Z)\colon S\to\c^{2n+1}$ $(n\in\n)$ is said to be a {\em generalized Legendrian curve} if $F\in\Ascr^1(S)$ and $F^*\eta_0=0$, where $\eta_0$ is the standard contact form of $\c^{2n+1}$ given in \eqref{eq:contact-C2n+1}; equivalently, if
\[
	dZ+\sum_{j=1}^n X_jdY_j=0\quad \text{everywhere on $S$}.
\]
\end{definition}

The following Mergelyan type approximation result for generalized Legendrian curves in $\c^{2n+1}$ is a particular instance of \cite[Theorem 5.1]{AlarconForstnericLopez2016Legendrian}.
\begin{theorem}\label{th:Mergelyan}
Let $S$ be a Runge admissible subset of an open Riemann surface $M$. Every generalized Legendrian curve $S\to\c^{2n+1}$ $(n\in\n)$ may be approximated in the $\Cscr^1(S)$-topology by holomorphic Legendrian embeddings $M\hra\c^{2n+1}$ having no constant component function.
\end{theorem}

The condition that the approximating embeddings in the above theorem can be chosen to do not have any constant component function is not explicitly mentioned in the statement of \cite[Theorem 5.1]{AlarconForstnericLopez2016Legendrian} but it easily follows from an inspection of its proof (see in particular \cite[Lemma 4.4 and proof of Lemma 5.2]{AlarconForstnericLopez2016Legendrian}).


\section{A general position result for Legendrian curves}\label{sec:GP}

In this section we prove a desingularizing result for Legendrian curves in $\c^{2n+1}$ which will be the key to ensure the one-to-oneness of the examples in Theorem \ref{th:intro-SL2C}.

\begin{lemma}\label{lem:GP}
Let $\cM=M\cup bM$ be a compact bordered Riemann surface. Every Legendrian curve $F\colon \cM\to\c^{2n+1}$ $(n\in\n)$ of class $\Ascr^1(\cM)$ may be approximated in the $\Cscr^1(\cM)$-topology by Legendrian embeddings $\wt F=(\wt F_1,\wt F_2,\ldots, \wt F_{2n+1})\colon \cM\hra\c^{2n+1}$ of class $\Ascr^1(\cM)$ having no constant component function and such that the map
\[
	(\wt F_1,\wt F_2,\ldots, e^{\wt F_{2n+1}})\colon \cM\to \c^{2n}\times\c_*\subset\c^{2n+1}
\]
is one-to-one.
\end{lemma}
Lemma \ref{lem:GP} is a subtle extension to \cite[Lemma 4.4]{AlarconForstnericLopez2016Legendrian}; it will enable us to construct embedded Legendrian curves in $\c^3$ which, by composing with the non-injective map $\Ygot\colon\c^3\to SL_2(\c)$ given in \eqref{eq:Ygot}, provide embedded Legendrian curves in $SL_2(\c)$.
\begin{proof}
For simplicity of exposition we shall assume that $n=1$; the same proof applies in general. Moreover, by Theorem \ref{th:Mergelyan} we may assume that the initial Legendrian curve $F$ is an embedding of class $\Ascr^1(\cM)$ having no constant component function. 

Let us write $F=(X,Y,Z)\colon \cM\to\c^3$. Consider the closed discrete subset
\begin{equation}\label{eq:Lambda}
	\Lambda:=\{(0,0,2m\pi\imath)\in\c^3\colon m\in\z\}\subset\c^3.
\end{equation}
Since $F$ is an embedding, the difference map $\delta F\colon \cM\times\cM\to\c^3$ defined by
\[
	\delta F(p,q)=F(q)-F(p),\quad p,q\in \cM,
\]
satisfies
\begin{equation}\label{eq:deltaF}
	(\delta F)^{-1}(0)=D_{\cM}:=\{(p,p)\colon p\in\cM\}.
\end{equation}
Thus, since $\delta F$ is continuous, $\Lambda$ is closed and discrete, and $\cM$ is compact, there is an open neighborhood $U\subset \cM\times\cM$ of the diagonal $D_{\cM}$ such that
\begin{equation}\label{eq:deltaFLambda}
\delta F(\overline U\setminus D_{\cM})\cap\Lambda=\emptyset.
\end{equation} 

On the other hand, by \cite[Proof of Lemma 4.4]{AlarconForstnericLopez2016Legendrian}, there exists a holomorphic map $H\colon \cM\times\c^N\to\c^3$ for some big $N\in\n$ such that the following conditions are satisfied for some $r>0$:
\begin{enumerate}[\it i)]
\item $H(\cdot,0)=F$.
\item $H(\cdot,\zeta)\colon \cM\to\c^3$ is a Legendrian immersion of class $\Ascr^1(\cM)$ for all $\zeta \in r\b$, where $\b$ denotes the unit ball in $\c^N$.
\item The difference map $\delta H\colon\cM\times\cM\times r\b\to\c^3$, defined by
\[
	\delta H(p,q,\zeta)=H(q,\zeta)-H(p,\zeta),\quad p,q\in\cM,\; \zeta\in r\b,
\]
is a submersive family of maps on $\cM\times\cM\setminus U$, meaning that the derivative
\[
	\di_\zeta|_{\zeta=0}\delta H(p,q,\zeta)\colon\c^N\to\c^3
\]
is surjective for all $(p,q)\in\cM\times\cM\setminus U$.
\end{enumerate}
Since $\cM\times\cM\setminus U$ is compact, {\it iii)} guarantees that the partial differential $\di_\zeta(\delta H)$ is surjective on $(\cM\times\cM \setminus U)\times r'\b$ for some number $0<r' <r$. It follows that the map $\delta H\colon (\cM\times\cM\setminus U)\times r'\b\to\c^3$ is transverse to any submanifold of $\c^3$, in particular, to the the closed discrete subset $\Lambda\subset\c^3$ \eqref{eq:Lambda}. By Abraham's reduction to Sard's theorem (see \cite{Abraham1963BAMS}; see also \cite[\textsection7.8]{Forstneric2011book} for the holomorphic case) we have that for a generic choice of $\zeta\in r'\b$ the difference map $\delta H(\cdot,\cdot,\zeta)$ is transverse to $\Lambda$ on $\cM\times\cM\setminus U$, and hence, by dimension reasons, 
\begin{equation}\label{eq:deltaFU}
	\delta H(\cM\times\cM\setminus U,\zeta)\cap\Lambda=\emptyset.
\end{equation}
Choosing $\zeta$ sufficiently close to $0\in\c^N$ we get a Legendrian immersion 
\[
	\wt F=(\wt X,\wt Y,\wt Z):=H(\cdot,\zeta)\colon\cM\to\c^3
\]
of class $\Ascr^1(\cM)$ which is as close to $F$ in the $\Cscr^1(\cM)$-norm as desired (in particular we may choose $\wt F$ having no constant component function), and, in view of \eqref{eq:deltaFLambda} and \eqref{eq:deltaFU}, satisfies
\[
	\delta \wt F(\cM\times\cM\setminus D_{\cM})\cap\Lambda=\emptyset.
\]
(Obviously, $\delta\wt F(p)=0\in\Lambda$ for all $p\in D_{\cM}$.)
This guarantees that the map $(\wt X,\wt Y,e^{\wt Z})\colon \cM\to\c^2\times\c_*\subset \c^3$ is one-to-one, and hence the same happens to $\wt F\colon \cM\to\c^3$. This implies that $\wt F$ is an embedding, which concludes the proof.
\end{proof}


\section{Proof of Theorem \ref{th:intro-SL2C}}\label{sec:SL2C}

The aim of this section is to prove Theorem \ref{th:intro-SL2C} in the introduction. Before proceeding with that, let us point out the following

\begin{remark}\label{rem:C3->SL2C}
Forstneri\v c, L\'opez, and the author proved in \cite[Theorem 1.1]{AlarconForstnericLopez2016Legendrian} that every open Riemann surface, $M$, carries a proper holomorphic Legendrian embedding $F=(F_1,F_2,F_3)\colon M\hra\c^3$, but also that, given $\{i,j\}\in\{1,2,3\}$, $i\neq j$, such an embedding $F$ can be found so that $(F_i,F_j)\colon M\to\c^2$ is a proper map. However, this fact does not imply Theorem \ref{th:intro-SL2C} (even allowing self-intersections) by making use of the map $\Ygot\colon\c^3\to SL_2(\c)$ given in \eqref{eq:Ygot}. Indeed, consider the sequence $\{p_j=(x_j,y_j,z_j)\}_{j\in\n}$ where $p_j=(e^{-j},-e^j,j)\in\c^3$ for all $j\in\n$. Observe that the sequences $\{(x_j,y_j)\}_{j\in\n}$, $\{(x_j,z_j)\}_{j\in\n}$, and $\{(y_j,z_j)\}_{j\in\n}$ are all divergent in $\c^2$, whereas
\[
	\Ygot(p_j)=\left(\begin{array}{cc}
	e^{-j} & 1
	\\
	-1 & 0
	\end{array}\right),\quad j\in\n,
\]
and so the sequence $\{\Ygot(p_j)\}_{j\in\n}$ is convergent (and hence bounded) in $SL_2(\c)$.
\end{remark}

We will obtain Theorem \ref{th:intro-SL2C} as a consequence of the following approximation result by proper (in a strong sense) Legendrian embeddings in $\c^{2n+1}$.
%
%
\begin{theorem}\label{th:main-SL2C}
Let $M$ be an open Riemann surface, $K\subset M$ be a smoothly bounded Runge compact domain, and $F=(F_1,F_2,\ldots,F_{2n+1})\colon K\to\c^{2n+1}$ be a Legendrian curve of class $\Ascr^1(K)$. Then $F$ may be approximated in the $\Cscr^1(K)$-topology by holomorphic Legendrian embeddings $\wt F=(\wt F_1,\wt F_2,\ldots,\wt F_{2n+1})\colon M\hra \c^{2n+1}$ satisfying the following properties:
\begin{enumerate}[\rm (i)]
\item The  function 
\[
	\max\{|\wt F_1|,-\Re\wt F_{2n+1}\}\colon M\to\r_+=[0,+\infty)
\] is proper,
where $\Re$ denotes the real part.
\item The holomorphic map 
\[
	(\wt F_1, \wt F_2,\ldots, e^{\wt F_{2n+1}})\colon M\to\c^{2n}\times\c_*\subset\c^{2n+1}
\]
is one-to-one.
\end{enumerate}
\end{theorem}

%
%

\begin{proof}[Proof of Theorem \ref{th:intro-SL2C} assuming Theorem \ref{th:main-SL2C}]
Let $M$ be an open Riemann surface. By Theorem \ref{th:main-SL2C} there is a holomorphic Legendrian embedding $F=(X,Y,Z)\colon M\hra \c^3$ such that 
\begin{enumerate}[\rm \it i)]
\item $\max\{|X|,-\Re Z\}\colon M\to\r_+$ is a proper map and
\item $(X,Y,e^Z)\colon M\to\c^3$ is one-to-one.
\end{enumerate}
Consider the holomorphic Legendrian immersion
\[
	\Ygot\circ F=\left(
	\begin{array}{cc}
	e^{-Z} & Xe^Z
	\\
	Ye^{-Z} & (1+XY)e^Z
	\end{array}\right)\colon M\to SL_2(\c),
\]
where $\Ygot\colon\c^3\to SL_2(\c)$ is the map \eqref{eq:Ygot}. It trivially follows from property {\it ii)} that $\Ygot\circ F$ is one-to-one, and hence, since proper injective immersions $M\hra SL_2(\c)$ are embeddings, to finish the proof it suffices to show that $\Ygot\circ F\colon M\hra SL_2(\c)$ is a proper map. Indeed, pick a divergent sequence $\{p_j\}_{j\in\n}$ in $M$ and let us check that $\{\Ygot(F(p_j))\}_{j\in\n}$ diverges in $SL_2(\c)\subset\c^4$; equivalently,
\begin{equation}\label{eq:divergent}
	\lim_{j\to\infty} |\Ygot(F(p_j))|_1=+\infty,
\end{equation}
where
\[
	|\Ygot(F(p))|_1 =  e^{-\Re Z(p)}(1+|Y(p)|)+e^{\Re Z(p)}(|X(p)|+|1+X(p)Y(p)|),\quad p\in M.
\]
In view of property {\it i)} we may assume that either $\lim_{j\to\infty} |X(p_j)|=+\infty$ or $\lim_{j\to\infty} \Re Z(p_j)=-\infty$; let us distinguish cases.

\smallskip

\noindent{\em Case 1. Assume that $\lim_{j\to\infty}\Re Z(p_j)=-\infty$.} It follows that
\[
	+\infty=\lim_{j\to\infty} e^{-\Re Z(p_j)}\leq \lim_{j\to\infty} |\Ygot(F(p_j))|_1,
\] 
which proves \eqref{eq:divergent}.

\smallskip

\noindent{\em Case 2. Assume that $\lim_{j\to\infty} |X(p_j)|=+\infty$.} We reason by contradiction and, up to passing to a subsequence, assume that $\{|\Ygot(F(p_j))|_1\}_{j\in\n}$ is a bounded sequence. It turns out that $\{e^{\Re Z(p_j)}|X(p_j)|\}_{j\in\n}$ is also bounded, and hence, since we are assuming that $\lim_{j\to\infty} |X(p_j)|=+\infty$, we infer that $\lim_{j\to\infty} e^{\Re Z(p_j)}=0$; equivalently, $\lim_{j\to\infty} \Re Z(p_j)=-\infty$. This reduces the proof to Case 1, and hence concludes the proof of the theorem.
\end{proof}

Theorem \ref{th:main-SL2C} will follow from a standard recursive application of Lemma \ref{lem:GP} and the following approximation result, which is the kernel of this section.

%
%
\begin{lemma}\label{lem:main-SL2C}
Let $M$ be an open Riemann surface, $\emptyset\neq K\Subset K'\subset M$ be smoothly bounded, Runge compact domains such that the Euler characteristic $\chi(K'\setminus\mathring K)\in\{-1,0\}$, $F=(F_1,F_2,\ldots,F_{2n+1})\colon K\to\c^{2n+1}$ be a Legendrian curve of class $\Ascr^1(K)$, 
$\rho>0$ be a positive number, and assume that
\begin{equation}\label{eq:lem-SL2C}
	\max\{|F_1|,-\Re F_{2n+1}\}>\rho\quad \text{everywhere on $bK$.}
\end{equation}
Then $F$ may be approximated in the $\Cscr^1(K)$-topology by Legendrian curves $\wt F=(\wt F_1,\wt F_2,\ldots,\wt F_{2n+1})\colon K'\to \c^{2n+1}$ of class $\Ascr^1(K')$ satisfying the following conditions:
\begin{enumerate}[\rm (i)]
\item $\max\{|\wt F_1|,-\Re\wt F_{2n+1}\}>\rho$ everywhere on $K'\setminus\mathring K$.
\item $\max\{|\wt F_1|,-\Re\wt F_{2n+1}\}>\rho+1$ everywhere on $bK'$.
\end{enumerate}
\end{lemma}
\begin{proof}
For simplicity of exposition we assume that $n=1$; the same proof applies in general. Write $F=(X,Y,Z)\colon K\to\c^3$. We distinguish cases.

\smallskip

\noindent{\em Case 1: Assume that $\chi(K'\setminus\mathring K)=0$.} In this case $K'\setminus \mathring K$ consists of finitely many pairwise disjoint compact annuli. Again for simplicity of exposition we assume that $K$ (and so $K'$) has a single boundary component, and hence $K'\setminus \mathring K$ is connected (an annulus); otherwise we would reason analogously on each connected component of $K'\setminus \mathring K$. 

By inequality \eqref{eq:lem-SL2C} there are an integer $m\ge 3$ and arcs $\alpha_j\subset bK$, $j\in\z_m=\z/m\z=\{0,\ldots,m-1\}$, meeting the following requirements:
\begin{enumerate}[\rm ({a}1)]
\item $\bigcup_{j\in\z_m}\alpha_j=bK$.
\item $\alpha_j\cap\alpha_{j+1}$ consists of a single point $p_j$ and $\alpha_j\cap\alpha_k=\emptyset$ for all $k\in\z_m\setminus\{j-1,j,j+1\}$, $j\in\z_m$.
\item There are disjoint subsets $I_X$ and $I_Z$ of $\z_m$ such that $I_X\cup I_Z=\z_m$, $|X|>\rho$ everywhere on $\alpha_j$ for all $j\in I_X$, and $-\Re Z>\rho$ everywhere on $\alpha_j$ for all $j\in I_Z$. 
\end{enumerate}

Let us now take a family of pairwise disjoint smooth Jordan arcs $\gamma_j\subset K'\setminus \mathring K$, $j\in\z_m$, having an endpoint $p_j\in bK$ and the other endpoint $q_j\in bK'$ and being otherwise disjoint from $bK\cup bK'$. We choose such arcs so that the compact set
\[
	S:=K\cup\big(\bigcup_{j\in\z_m} \gamma_j\big)\subset K'
\]
is admissible in $M$ in the sense of Def. \ref{def:admissible}. It follows that $\mathring K'\setminus S$ consists of $m$ pairwise disjoint disks; we denote by $\Omega_j$ the one whose closure contains $\alpha_j$, and by $\beta_j$ the arc $bK'\cap\overline\Omega_j$. (See Figure \ref{fig:Omegaj}.)
Thus, 
\begin{equation}\label{eq:bK'}
	K'\setminus\mathring K=\bigcup_{j\in\z_m}\overline \Omega_j,\quad bK'=\bigcup_{j\in\z_m}\beta_j,
\end{equation}
and
\[
	b\Omega_j=\overline\Omega_j\setminus\Omega_j=
	\gamma_{j-1}\cup\alpha_j\cup\gamma_j\cup\beta_j,\quad j\in\z_m.
\]
\begin{figure}[ht]
    \begin{center}
    \scalebox{0.12}{\includegraphics{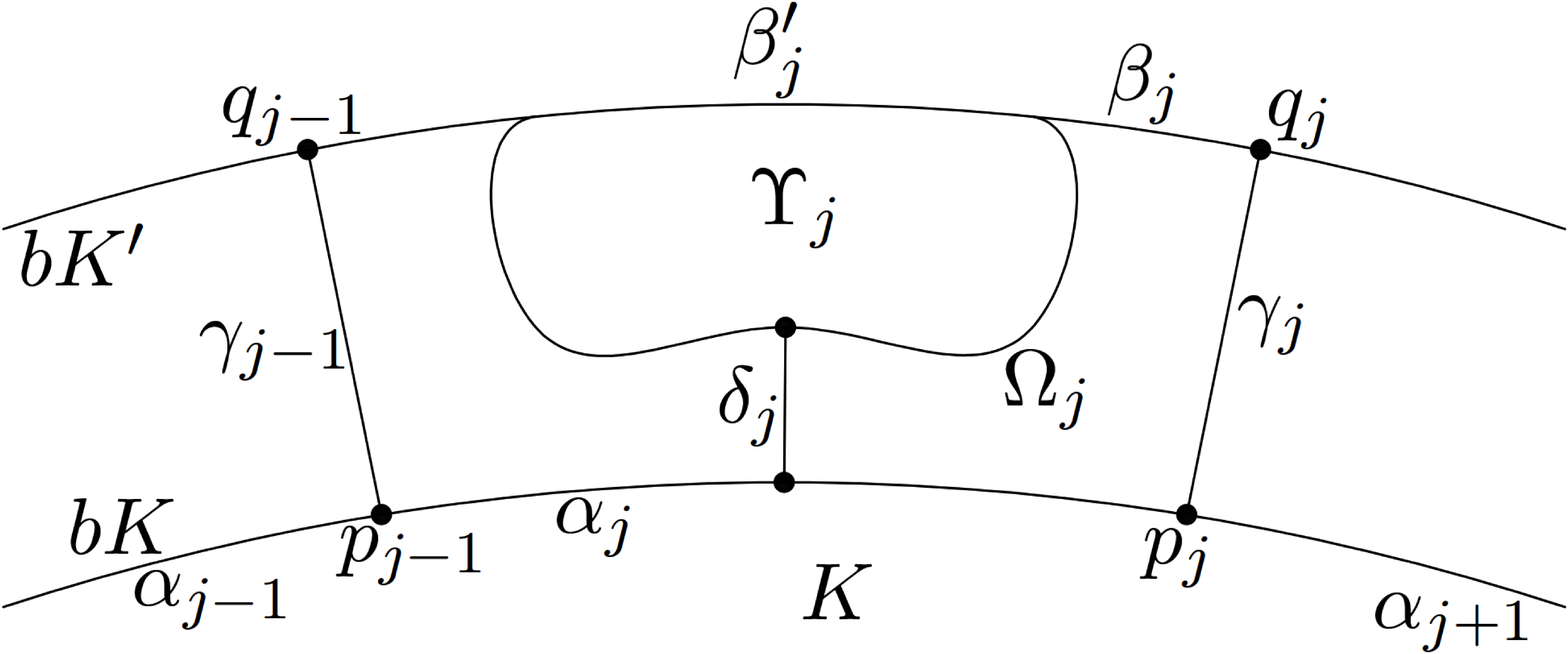}}
        \end{center}
\caption{$K'\setminus\mathring K$.}\label{fig:Omegaj}
\end{figure}

Since every compact path $[0,1]\to\c^3$ may be uniformly approximated by Legendrian paths (cf.\ \cite[Theorem A.6]{AlarconForstnericLopez2016Legendrian}), property {\rm (a3)} enables us to extend $F=(X,Y,Z)$, with the same name, to a generalized Legendrian curve $S\to\c^3$ (see Def.\ \ref{def:generalized}) satisfying:
\begin{enumerate}[\rm ({b}1)]
\item $|X|>\rho$ everywhere on $\gamma_{j-1}\cup\alpha_j\cup\gamma_j $ and $|X|>\rho+1$ at $q_{j-1}$ and $q_j$ for all $j\in I_X$.
\item $-\Re Z>\rho$ everywhere on $\gamma_{j-1}\cup\alpha_j\cup\gamma_j $ and $-\Re Z>\rho+1$ at $q_{j-1}$ and $q_j$ for all $j\in I_Z$.
\end{enumerate}

Next, by Theorem \ref{th:Mergelyan}, we may approximate $F$ in the $\Cscr^1(S)$-topology by Legendrian curves $K'\to\c^3$ of class $\Ascr^1(K')$ having no constant component function. We still denote by $F=(X,Y,Z)$ to a such approximating curve and assume that the approximation is close enough so that properties {\rm (b1)} and {\rm (b2)} remain to hold. Thus, by continuity of $F$, for each $j\in\z_m$ there is a smoothly bounded, closed disk 
\[
	\Upsilon_j\subset\overline\Omega_j\setminus (\gamma_{j-1}\cup\alpha_j\cup\gamma_j )= 
	\Omega_j\cup(\beta_j\setminus\{q_{j-1},q_j\})
\] 
such that:
\begin{enumerate}[\rm ({c}1)]
\item $\beta_j':=\Upsilon_j\cap\beta_j\neq\emptyset$ is an arc contained in the relative interior of $\beta_j$.
\item $|X|>\rho$ everywhere on $\overline{\Omega_j\setminus\Upsilon_j}$ and $|X|>\rho+1$ everywhere on $\overline{\beta_j\setminus\beta_j'}$ for all $j\in I_X$.
\item $-\Re Z>\rho$ everywhere on $\overline{\Omega_j\setminus\Upsilon_j}$ and $-\Re Z>\rho+1$ everywhere on $\overline{\beta_j\setminus\beta_j'}$ for all $j\in I_Z$.
\end{enumerate}
See Figure \ref{fig:Omegaj}.

Let us now assume that $I_X\neq\emptyset$; otherwise $I_Z=\z_m$ and the following deformation procedure is not required. For each $j\in I_X$ choose an arc $\delta_j\subset\overline{\Omega_j\setminus\Upsilon_j}$ having an endpoint in $\Upsilon_j\setminus\beta_j$ and the other endpoint in the relative interior of $\alpha_j$, and being otherwise disjoint from $\Upsilon_j\cup b\Omega_j$. (See Figure \ref{fig:Omegaj}.) Further, we may choose the arcs $\delta_j$, $j\in\z_m$, so that
\[
	S_X:=S\cup\big(\bigcup_{j\in I_Z}\overline\Omega_j\big)\cup
	\big(\bigcup_{j\in I_X}\delta_j\cup\Upsilon_j\big)
\] 
is admissible in $M$ and, taking into account that $F$ is of class $\Ascr^1(K')$ and has no constant component function, $X$, $Y$, and $Z$ have neither zeros nor critical points in $\bigcup_{j\in I_X}\delta_j$. Let $G_1=(X_1,Y_1,Z_1)\colon S_X\to\c^3$ be a generalized Legendrian curve enjoying the following properties:
\begin{enumerate}[\rm ({d}1)]
\item $G_1=F$ everywhere on $S\cup\big(\bigcup_{j\in I_Z}\overline\Omega_j\big)$.
\item $X_1=X$ everywhere on $S_X$.
\item $-\Re Z_1>\rho+1$ everywhere on $\bigcup_{j\in I_X}\Upsilon_j$.
\end{enumerate}
Such a generalized Legendrian curve can be constructed as follows. Set $X_1:=X|_{S_X}$, hence {\rm (d2)} holds, and choose any map $Z_1\colon S_X\to\c$ of class $\Ascr^1(S_X)$ meeting the following requirements:
\begin{enumerate}[\it i)]
\item $Z_1=Z$ everywhere on $S\cup\big(\bigcup_{j\in I_Z}\overline\Omega_j\big)$.
\item $-\Re Z_1>\rho+1$ everywhere on $\bigcup_{j\in I_X}\Upsilon_j$.
\item $dZ_1$ (cf.\ \eqref{eq:df}) has no zeros on $\bigcup_{j\in I_X}\delta_j$ and its zeros on $\bigcup_{j\in I_X}\Upsilon_j$ are those of $X_1$, with the same order.
\end{enumerate}
Such a function trivially exists; to ensure {\it iii)} recall that $X_1=X$ and $Z$ have neither zeros nor critical points in $\bigcup_{j\in I_X}\delta_j$. Now fix a point $u_0\in \mathring K$ and define $Y_1\colon S_X\to\c$ by
\[
	S_X\ni u\longmapsto Y_1(u)=\left\{
	\begin{array}{ll}
	Y(u) & \text{for all $u\in S\cup\big(\bigcup_{j\in I_Z}\overline\Omega_j\big)$}
	\\
	\displaystyle Y(u_0)-\int_{u_0}^u\frac{dZ_1}{X_1} & 
	\text{for all $u\in \bigcup_{j\in I_X}\delta_j\cup\Upsilon_j$}.
	\end{array}\right.
\]
By {\it i)}, {\rm (d2)}, {\it iii)}, and the facts that $F=(X,Y,Z)\colon K'\to\c^3$ is Legendrian, that $S_X$ is a strong deformation retract of $K'$, that $X$ vanishes nowhere on $\bigcup_{j\in I_X}\delta_j$, and that $\delta_j\cup\Upsilon_j$ is simply-connected for all $j\in I_X$, we infer that $Y_1\colon S_X\to\c$ is a well-defined function of class $\Ascr^1(S_X)$ and, taking also {\it ii)} into account, the map $G_1:=(X_1,Y_1,Z_1)\colon S_X\to\c^3$ is a generalized Legendrian curve satisfying conditions {\rm (d1)} and {\rm (d3)}.

In view of {\rm (d2)} we have that $X_1=X$ is nonconstant and of class $\Ascr^1(K')$, and so \cite[Lemma 4.3]{AlarconForstnericLopez2016Legendrian} guarantees that $G_1$ may be approximated in the $\Cscr^1(S_X)$-topology by Legendrian curves $\wt G_1=(\wt X_1,\wt Y_1,\wt Z_1)\colon K'\to\c^3$ of class $\Ascr^1(K')$ having no constant component function and with 
\begin{equation}\label{eq:wtX1=X}
	\wt X_1=X.
\end{equation}
Further, in view of {\rm (c3)}, {\rm (d1)}, and {\rm (d3)}, we may choose such an approximation $\wt G_1$ of $G_1$ so that:
\begin{enumerate}[\rm ({e}1)]
\item $-\Re\wt Z_1>\rho$ everywhere on $\bigcup_{j\in I_Z} \overline{\Omega_j\setminus\Upsilon_j}$.
\item $-\Re\wt Z_1>\rho+1$ everywhere on $\bigcup_{j\in I_X} \Upsilon_j$ and on $\bigcup_{j\in I_Z} \overline{\beta_j\setminus\beta_j'}$ (for the latter take into account that $\beta_j\subset \overline\Omega_j$ for all $j\in\z_m\supset I_Z$).
\end{enumerate}

Assume for a moment that $I_Z=\emptyset$ and let us show that $\wt F:=\wt G_1$ solves the lemma. Indeed, in this case $I_X=\z_m$ and hence, in view of \eqref{eq:bK'},
\begin{equation}\label{eq:IZ=empty}
	K'\setminus\mathring K=\bigcup_{j\in I_X}\overline\Omega_j=\bigcup_{j\in I_X} 
	\Upsilon_j\cup\overline{\Omega_j\setminus\Upsilon_j},\quad
	bK'=\bigcup_{j\in I_X}\beta_j\subset \bigcup_{j\in I_X}\Upsilon_j\cup \overline{\beta_j\setminus\beta_j'}.
\end{equation}
On the other hand, \eqref{eq:wtX1=X}, {\rm (c2)}, and {\rm (e2)} ensure that
\begin{equation}\label{eq:max}
	\max\{|\wt X_1|,-\Re\wt Z_1\}>\rho\quad  \text{everywhere on }\bigcup_{j\in I_X} 
	\Upsilon_j\cup\overline{\Omega_j\setminus\Upsilon_j} 
\end{equation}
and
\begin{equation}\label{eq:max1}
	\max\{|\wt X_1|,-\Re\wt Z_1\}>\rho+1\quad \text{everywhere on }\bigcup_{j\in I_X}\Upsilon_j\cup \overline{\beta_j\setminus\beta_j'}.
\end{equation}
Thus, \eqref{eq:IZ=empty}, \eqref{eq:max}, and \eqref{eq:max1} guarantee conditions {\rm (i)} and {\rm (ii)} in the statement of the lemma. Moreover, if the approximation of $G_1$ by $\wt G_1$ is close enough in the $\Cscr^1(S_X)$-norm, property {\rm (d1)} and the fact that $K\subset S\subset S_X$ enable us to assume that $\wt G_1$ is as close as desired to $F$ in the $\Cscr^1(K)$-topology. This would conclude the proof of the lemma in case $I_Z=\emptyset$.

Assume now that $I_Z\neq\emptyset$. Analogously to what has been done in the previous deformation procedure, for each $j\in I_Z$ we choose an arc $\delta_j\subset\overline{\Omega_j\setminus\Upsilon_j}$ having an endpoint in $\Upsilon_j\setminus\beta_j$ and the other endpoint in the relative interior of $\alpha_j$, being otherwise disjoint from $\Upsilon_j\cup b\Omega_j$, and such that the set
\[
	S_Z:=S\cup\big(\bigcup_{j\in I_X}\overline\Omega_j\big)\cup
	\big(\bigcup_{j\in I_Z}\delta_j\cup\Upsilon_j\big)
\] 
is admissible in $M$ and the functions $\wt X_1$, $\wt Y_1$, and $\wt Z_1$ have neither zeros nor critical points in $\bigcup_{j\in I_Z}\delta_j$; see Figure \ref{fig:Omegaj}. For the latter, recall that the concerned functions are nonconstant and of class $\Ascr^1(K')$. Let $G_2=(X_2,Y_2,Z_2)\colon S_Z\to\c^3$ be a generalized Legendrian curve satisfying the following properties:
\begin{enumerate}[\rm ({f}1)]
\item $G_2=\wt G_1$ everywhere on $S\cup\big(\bigcup_{j\in I_X}\overline\Omega_j\big)$.
\item $Z_2=\wt Z_1$ everywhere on $S_Z$.
\item $|X_2|>\rho+1$ everywhere on $\bigcup_{j\in I_Z}\Upsilon_j$.
\item $X_2$ and $Y_2$ are nonconstant on $\Upsilon_j$, $X_2$ has no zeros in $\delta_j$, and $Y_2$ has no critical points in $\delta_j$ for all $j\in I_Z$. 
\end{enumerate}
Such may be constructed as follows. Set $Z_2:=\wt Z_1|_{S_Z}$; this implies {\rm (f2)}. Choose any map $X_2\colon S_Z\to\c$ of class $\Ascr^1(S_Z)$ satisfying the following conditions:
\begin{enumerate}[\it I)]
\item $X_2=\wt X_1$ everywhere on $S\cup\big(\bigcup_{j\in I_X}\overline \Omega_j\big)$.
\item $|X_2|>\rho+1$ everywhere on $\bigcup_{j\in I_Z}\Upsilon_j$ and $X_2$ is nonconstant on $\Upsilon_j$ for all $j\in I_Z$.
\item $X_2$ vanishes nowhere on $\bigcup_{j\in I_Z}\delta_j$.
\end{enumerate}
Existence of such a function is clear; recall that $\wt X_1$ does not vanish anywhere on $\bigcup_{j\in I_Z}\delta_j$. Fix a point $u_0\in\mathring K$ and define $Y_2\colon S_Z\to\c$ by
\[
	S_Z\ni u\longmapsto Y_2(u)=\left\{
	\begin{array}{ll}
	\wt Y_1(u) & \text{for all $u\in S\cup\big(\bigcup_{j\in I_X}\overline\Omega_j\big)$}
	\\
	\displaystyle \wt Y_1(u_0)-\int_{u_0}^u\frac{dZ_2}{X_2} & 
	\text{for all $u\in \bigcup_{j\in I_Z}\delta_j\cup\Upsilon_j$}.
	\end{array}\right.
\]
In view of {\rm (f2)}, {\it I)}, {\it II)}, {\it III)}, the facts that $\wt G_1$ is a Legendrian curve of class $\Ascr^1(K')$, that $S_Z$ is a strong deformation retract of $K'$, and that $\delta_j\cup\Upsilon_j$ is simply-connected for all $j\in I_Z$, imply that $Y_2$ is a well-defined map of class $\Ascr^1(S_Z)$ (observe that $X_2$ vanishes nowhere on $\bigcup_{j\in I_Z}\Upsilon_j$ since $\rho>0$) and $G_2:=(X_2,Y_2,Z_2)\colon S_Z\to\c^3$ is a generalized Legendrian curve satisfying properties {\rm (f1)} and {\rm (f3)}. Finally, since $\wt Z_1$ has no critical points in $\delta_j$ and is nonconstant on $\Upsilon_j$ for all $j\in I_Z$, properties {\it II)}, {\it III)}, and {\rm (f2)}  guarantee {\rm (f4)}. 

Now we may apply \cite[Lemma 4.3]{AlarconForstnericLopez2016Legendrian} to $G_2$ inferring that it may be approximated in the $\Cscr^1(S_Z)$-topology by Legendrian curves $\wt F=(\wt X,\wt Y,\wt Z)\colon K'\to\c^3$ of class $\Ascr^1(K')$ having no constant component function and satisfying
\begin{equation}\label{eq:wtZ=wtZ1}
	\wt Z=\wt Z_1.
\end{equation} 

We claim that a close enough such approximation $\wt F$ of $G_2$ satisfies the conclusion of the lemma. Indeed, since $K\subset S\subset S_X\cap S_Z$ then, by {\rm (d1)}, {\rm (f1)}, and choosing $\wt G_1$ close enough to $G_1$ in the $\Cscr^1(S_X)$-norm and $\wt F$ close enough to $G_2$ in the $\Cscr^1(S_Z)$-norm, we may choose $\wt F$ to be as close as desired to $F$ in the $\Cscr^1(K)$-norm. On the other hand, {\rm (e1)}, the second part of {\rm (e2)}, {\rm (f3)}, and \eqref{eq:wtZ=wtZ1} give that
\begin{equation}\label{eq:max2}
	\max\{|\wt X|,-\Re\wt Z\}>\rho\quad  \text{everywhere on }\bigcup_{j\in I_Z} 
	\Upsilon_j\cup\overline{\Omega_j\setminus\Upsilon_j} 
\end{equation}
and
\begin{equation}\label{eq:max3}
	\max\{|\wt X|,-\Re\wt Z\}>\rho+1\quad \text{everywhere on }\bigcup_{j\in I_Z}\Upsilon_j\cup \overline{\beta_j\setminus\beta_j'},
\end{equation}
provided that $\wt X$ is chosen close enough to $X_2$ uniformly on $S_Z\supset \bigcup_{j\in I_Z} \Upsilon_j$. Finally, since $\z_m=I_X\cup I_Z$,
\[
	K'\setminus\mathring K=\bigcup_{j\in \z_m}\overline\Omega_j=\bigcup_{j\in \z_m} 
	\Upsilon_j\cup\overline{\Omega_j\setminus\Upsilon_j},\quad\text{and}\quad
	bK'=\bigcup_{j\in \z_m}\beta_j\subset \bigcup_{j\in \z_m}\Upsilon_j\cup \overline{\beta_j\setminus\beta_j'},
\]
properties \eqref{eq:max}, \eqref{eq:max1}, {\rm (f1)}, \eqref{eq:wtZ=wtZ1}, \eqref{eq:max2}, and \eqref{eq:max3} guarantee conditions {\rm (i)} and {\rm (ii)} in the statement of the lemma, whenever that the approximation of $X_2$ by $\wt X$ on the set $S_Z\supset \bigcup_{j\in I_X}\overline\Omega_j=\bigcup_{j\in I_X} \Upsilon_j\cup\overline{\Omega_j\setminus\Upsilon_j}$ is sufficiently close. This concludes the proof in the case when the Euler characteristic $\chi(K'\setminus\mathring K)=0$.

\smallskip

\noindent{\em Case 2: Assume that $\chi(K'\setminus\mathring K)=-1$}. In this case there is a Jordan arc $\gamma\subset \mathring K'\setminus\mathring K$  such that the two endpoints of $\gamma$ lie in $bK$ and $\gamma$ is otherwise disjoint from $K$, and $S:=K\cup\gamma\Subset K'$ is Runge and admissible in $M$ (in the sense of Def.\ \ref{def:admissible}) and a strong deformation retract of $K'$. Since every compact path in $\c^3$ may be uniformly approximated by Legendrian paths (see \cite[Theorem A.6]{AlarconForstnericLopez2016Legendrian}), inequality \eqref{eq:lem-SL2C} enables us to extend $F$, with the same name, to a generalized Legendrian curve $S\to\c^3$ such that
\begin{equation}\label{eq:lemma-chi-1}
	\max\{|X|,-\Re Z\}>\rho\quad \text{everywhere on $bK\cup\gamma$}.
\end{equation}

Now, by Theorem \ref{th:Mergelyan}, we may approximate $F$ uniformly in the $\Cscr^1(S)$-topology by holomorphic Legendrian curves $F_1=(X_1,Y_1,Z_1)\colon M\to \c^3$. Since $S$ is a strong deformation retract of $K'$ then, if the approximation of $F$ by $F_1$ is close enough, \eqref{eq:lemma-chi-1} ensures the existence of a smoothly bounded Runge compact domain $K''$ such that $S\Subset K''\Subset K'$, the Euler characteristic $\chi(K'\setminus\mathring K'')=0$, and $\max\{|X_1|,-\Re Z_1\}>\rho$ everywhere on $K''\setminus\mathring K$. This reduces the proof to Case 1, and hence concludes the proof of the lemma.
\end{proof}

To finish the proof of Theorem \ref{th:intro-SL2C}, we now prove the main result of this section.
%
%
\begin{proof}[Proof of Theorem \ref{th:main-SL2C}]
For simplicity of exposition we assume that $n=1$ and write $F=(X,Y,Z)$; the same proof applies in general. By approximation, we may also assume in view of Theorem \ref{th:Mergelyan} that $F$ extends, with the same name, to a holomorphic Legendrian embedding on an open neighborhood of $K$ having no constant component functions. Thus, up to slightly enlarging $K$ if necessary, we may assume that $X$ does not vanish anywhere on $bK$, and hence there is a number $\rho_0>0$ such that
\begin{equation}\label{eq:basis}
	\max\{|X|,-\Re Z\}>\rho_0\quad \text{everywhere on $bK$}.
\end{equation}

Let
\[
	K_0:=K\Subset K_1\Subset K_2\Subset \cdots \Subset \bigcup_{j\in\z_+} K_j=M
\]
be an exhaustion of $M$ by smoothly bounded, Runge compact domains such that the Euler characteristic $\chi(K_j\setminus \mathring K_{j-1})\in\{-1,0\}$ for all $j\in \n$. The existence of such an exhaustion is well known; see for instance \cite[Lemma 4.2]{AlarconLopez2013JGA} for a simple proof.

Set $F_0:=F$. For any sequence of positive numbers $\{\epsilon_j\}_{j\in\n}\searrow 0$, a standard recursive application of Lemma \ref{lem:main-SL2C} provides a sequence of Legendrian curves $\{F_j=(X_j,Y_j,Z_j)\colon K_j\to\c^3\}_{j\in\n}$ of class $\Ascr^1(K_j)$ such that the following conditions hold for all $j\in\n$:
\begin{enumerate}[\rm (a)]
\item $\|F_j-F_{j-1}\|_{1,K_{j-1}}<\epsilon_j$.
\item $\max\{|X_j|,-\Re Z_j\}>j-1$ everywhere on $K_j\setminus\mathring K_{j-1}$.
\item $\max\{|X_j|,-\Re Z_j\}>j$ everywhere on $bK_j$.
\end{enumerate}
Furthermore, by Lemma \ref{lem:GP} we may also assume that
\begin{enumerate}[\rm (a)]
\item[\rm (d)] $(X_j,Y_j,e^{Z_j})\colon K_j\to\c^3$ is one-to-one for all $j\in\n$.
\end{enumerate}

Thus, choosing the number $\epsilon_j>0$ small enough at each step in the recursive construction, {\rm (a)} and {\rm (d)} ensure that the sequence $\{F_j\}_{j\in\n}$ converges uniformly on compact subsets of $M$ to a holomorphic Legendrian immersion $\wt F=(\wt X,\wt Y,\wt Z)\colon M\to \c^3$ which is as close as desired to $F_0=F$ in the $\Cscr^1$-norm on $K_0=K$ and such that the holomorphic map $(\wt X,\wt Y,e^{\wt Z})\colon M\to \c^2\times\c_*\subset\c^3$ is one-to-one. It follows that $\wt F\colon M\to\c^3$ is one-to-one as well. Moreover, if each $\epsilon_j>0$ is chosen sufficiently small, condition {\rm (b)} guarantees that $\max\{|\wt X|,-\Re\wt Z\}\colon M\to\r_+$ is a proper map, and hence $\wt F$ is an embedding. This concludes the proof.
\end{proof}


\subsection*{Acknowledgements}
The author is supported by the Ram\'on y Cajal program of the Spanish Ministry of Economy and Competitiveness and partially supported by the MINECO/FEDER grant no. MTM2014-52368-P, Spain.

I thank Franc Forstneri\v c and Antonio Mart\'inez for helpful suggestions which led to improvement of the paper.




\vspace*{0.3cm}
\noindent Antonio Alarc\'{o}n

\noindent Departamento de Geometr\'{\i}a y Topolog\'{\i}a e Instituto de Matem\'aticas (IEMath-GR), Universidad de Granada, Campus de Fuentenueva s/n, E--18071 Granada, Spain.

\noindent  e-mail: {\tt alarcon@ugr.es}

\end{document}